\documentclass[a4paper,12pt,reqno]{amsart}
\usepackage{a4wide}
\usepackage{amsmath}
\usepackage{amssymb}
\usepackage{amsthm}
\usepackage{latexsym}
\usepackage{graphicx}
\usepackage[english]{babel}
              
\newtheorem{prop}[subsection]{Proposition}

\newtheorem{teor}[subsection]{Theorem}
\newtheorem{lema}[subsection]{Lemma}

\theoremstyle{definition}
\newtheorem{dfn} [subsection]{Definition}
\theoremstyle{remark}

\newtheorem{exm} [subsection]{Example}

\newcommand{\paa}{p_{\mathbf a}}
\newcommand{\Pa}{P_{\mathbf a}}

\newcommand{\stir}{\genfrac{[}{]}{0pt}{}}

\def\pre{\operatorname{pre}}

\def\lcm{\operatorname{lcm}}
\def\gcd{\operatorname{gcd}}

\def\Exp{\operatorname{Exp}}
\def\Log{\operatorname{Log}}
\def\pp{\operatorname{p}}

\numberwithin{equation}{section}

\begin{document}

\title[Remarks on $d$-ary partitions and an application to elementary symmetric partitions]{Remarks on $d$-ary partitions and an application to elementary symmetric partitions}
\author[Mircea Cimpoea\c s, Roxana T\u anase 
       ]{Mircea Cimpoea\c s$^1$ and Roxana T\u anase$^2$}  
\date{}

\keywords{Restricted partitions, $d$-ary partitions, Elementary symmetric partitions}

\subjclass[2020]{11P81, 11P83}

\footnotetext[1]{ \emph{Mircea Cimpoea\c s}, National University of Science and Technology Politehnica Bucharest, Faculty of
Applied Sciences, 
Bucharest, 060042, Romania and Simion Stoilow Institute of Mathematics, Research unit 5, P.O.Box 1-764,
Bucharest 014700, Romania, E-mail: mircea.cimpoeas@upb.ro,\;mircea.cimpoeas@imar.ro}
\footnotetext[2]{ \emph{Roxana T\u anase}, National University of Science and Technology Politehnica Bucharest, Faculty of
Applied Sciences, 
Bucharest, 060042, E-mail: roxana\_elena.tanase@upb.ro}


\begin{abstract}
We prove new formulas for $\pp_d(n)$, the number of $d$-ary partitions of $n$, and, also, 
for $P_d(n)$, its polynomial part. 

Given a partition $\lambda=(\lambda_1,\ldots,\lambda_{\ell})$, its associated 
$j$-th symmetric elementary partition, $\pre_{j}(\lambda)$, is
the partition whose parts are $\{\lambda_{i_1}\cdots\lambda_{i_j}\;:\;1\leq i_1 < \cdots < i_j\leq \ell\}$.
We prove that if $\lambda$ and $\mu$ are two $d$-ary partitions of length $\ell$ such that 
$\pre_j(\lambda)=\pre_j(\mu)$ and $\lambda_{i_1}\cdots \lambda_{i_j} = 
\mu_{i_1}\cdots \mu_{i_j}$, for all $1\leq i_1 < \cdots < i_j\leq \ell$, then $\lambda=\mu$.
\end{abstract}

\maketitle

\section{Introduction}

Let $n$ be a positive integer. We denote $[n]=\{1,2,\ldots,n\}$.
A partition of $n$ is a non-increasing sequence of positive integers $\lambda_i$ whose sum equals $n$. We define
$p(n)$ as the number of partitions of $n$ and we define $p(0) = 1$. We denote $\lambda=(\lambda_1,\lambda_2,\ldots,\lambda_{\ell})$
with $\lambda_1\geq \lambda_2\geq \cdots \geq \lambda_{\ell} \geq 1$ and $|\lambda|:=\lambda_1+\cdots+\lambda_{\ell}=n$.
We refer to $|\lambda|$ as the size of $\lambda$ and the numbers $\lambda_i$ as parts of $\lambda$. The number $\ell(\lambda)=\ell$
is the number of parts of $\lambda$ and it is called the length of $\lambda$.
For more on the theory of partitions, we refer the reader to \cite{andrews}.

Let $d\geq 2$ be an integer. A partition $\lambda=(\lambda_1,\ldots,\lambda_{\ell})$ is called \emph{$d$-ary}, if all
$\lambda_i$'s are powers of $d$. A $2$-ary partition is called binary. In Proposition \ref{p1} we establish a natural bijection
between the set of all integer partition and the set of $d$-ary partitions, which conserves the length (but not the size).

In Theorem \ref{t1} we give a new formula for $p_d(n)$, the number of $d$-ary partitions of $n$, using the fact that a $d$-ary
partition is a partition with the parts in $\{1,d,d^2,d^3,\ldots\}$. In Theorem \ref{t2}, we give a new formula for $W_j(d,n)$'s, the
Sylvester waves of $p_d(n)$. Also, in Theorem \ref{t3} and Theorem \ref{t4} we give new formulas for $P_d(n)=W_1(d,n)$, the polynomial
part of $p_d(n)$. In Example \ref{exe}, we illustrate these results.


Now, let $K$ be an arbitrary field and $S=K[x_1,\ldots,x_{\ell}]$ be the ring of polynomials over $K$ in $\ell$ indeterminates.
We recall that the $j^{th}$ elementary symmetric polynomial of $S$ is 
$$e_j(x_1,\ldots,x_{\ell})=\sum_{1\leq i_1 < i_2 <\ldots <i_j\leq \ell}x_{i_1}x_{i_2}\cdots x_{i_j},\text{ where }1\leq j\leq \ell.$$
Also, we define $e_0(x_1,\ldots,x_{\ell})=1$ and $e_j(x_1,\ldots,x_{\ell})=0$ for $j>\ell$.

Given a partition $\lambda$, we have $e_j(\lambda)=0$ if $\ell(\lambda)<j$ and 
$$e_j(\lambda)=\sum_{1\leq i_1 < i_2 <\ldots <i_j\leq \ell(\lambda)}\lambda_{i_1}\lambda_{i_2}\cdots \lambda_{i_j},\text{ if }1\leq j\leq \ell(\lambda).$$
In \cite{merca,merca2}, Ballantine et al. introduced the following definition. Given a partition $\lambda$, the partition $\pre_{j}(\lambda)$ is
the partition whose parts are $$\{\lambda_{i_1}\cdots\lambda_{i_j}\;:\;1\leq i_1 < i_2 < \cdots < i_j\leq \ell(\lambda)\},$$
and they called it an \emph{elementary symmetric partition}.
Note that $\pre_1(\lambda)=\lambda$, but $\pre_j(\lambda)\neq \lambda$, for $j\geq 2$.
For example, if $\lambda=(3,2,1,1)$, then $\pre_{2}(\lambda)=(6,3,3,2,2,1)$.


A natural question to ask is the following: If $\lambda$ and $\mu$ are two partitions such that $\pre_j(\lambda)=\pre_j(\mu)$ then
is it true that $\lambda=\mu$? Only the following cases are known in literature: (i) $j=2$ and $m(\lambda),m(\mu)\leq 3$, see 
\cite[Proposition 14]{merca2} and (ii) $j=2$ and $\lambda$ and $\mu$ are binary partitions; see 
\cite[Proposition 15]{merca2}. In Theorem \ref{tp} 
we prove that if 
$\lambda$ and $\mu$ are two $d$-ary partitions of length $\ell$ such that $\pre_j(\lambda)=\pre_j(\mu)$ and $\lambda_{i_1}\cdots \lambda_{i_j} = 
\mu_{i_1}\cdots \mu_{i_j}$, for all $1\leq i_1 < \cdots < i_j\leq \ell$, where $1\leq j\leq \ell-1$, then $\lambda=\mu$.

\section{Preliminaries}

Let $\mathbf a := (a_1, a_2, \ldots , a_r)$ be a sequence of positive integers, where $r \geq 1$. 
Let $\lambda$ be a partition. We say that $\lambda$ has parts in $\mathbf a$ if $\lambda_i \in \{a_1,\ldots,a_r\}$ for
all $1\leq i\leq \ell(\lambda)$. 

The \emph{restricted partition
function} associated to $\mathbf a$ is $\paa : \mathbb N \to \mathbb N$, $\paa(n) :=$ the number of integer solutions $(x_1, \ldots, x_r)$
of $\sum_{i=1}^r a_ix_i = n$ with $x_i \geq 0$. In other words, $\paa(n)$ counts the number of partitions of $n$ with parts in $\mathbf a$.
Note that the generating function of $\paa(n)$ is
\begin{equation}\label{gen}
\sum_{n=0}^{\infty}\paa(n)z^n= \frac{1}{(1-z^{a_1})\cdots(1-z^{a_r})}.
\end{equation}
Let $D$ be a common multiple of $a_1$, $a_2,\ldots,a_r$. 
Bell \cite{bell} proved that $\paa(n)$ is a quasi-polynomial of degree $k-1$, with the period $D$, that is
\begin{equation}\label{quasi}
\paa(n)=d_{\mathbf a,k-1}(n)n^{k-1}+\cdots+d_{\mathbf a,1}(n)n+d_{\mathbf a,0}(n), 
\end{equation}
where $d_{\mathbf a,m}(n+D)=d_{\mathbf a,m}(n)$ for $0\leq m\leq k-1$ and $n\geq 0$, and $d_{\mathbf a,k-1}(n)$ is
not identically zero.

Sylvester \cite{sylvester},\cite{sylv} decomposed the restricted partition in a sum of ``waves'': 
\begin{equation}\label{wave}
\paa(n)=\sum_{j\geq 1} W_{j}(n,\mathbf a), 
\end{equation}
where the sum is taken over all distinct divisors $j$ of the components of $\mathbf a$ and showed that for each such $j$, 
$W_j(n,\mathbf a)$ is the coefficient of $t^{-1}$ in
$$ \sum_{0 \leq \nu <j,\; \gcd(\nu,j)=1 } \frac{\rho_j^{-\nu n} e^{nt}}{(1-\rho_j^{\nu a_1}e^{-a_1t})\cdots (1-\rho_j^{\nu a_k}e^{-a_kt}) },$$
where $\rho_j=e^{\frac{2\pi i}{j}}$ and $\gcd(0,0)=1$ by convention. Note that $W_{j}(n,\mathbf a)$'s are quasi-polynomials of period $j$.
Also, $W_1(n,\mathbf a)$ is called the \emph{polynomial part} of $\paa(n)$ and it is denoted by $\Pa(n)$.


\begin{teor}(\cite[Corollary 2.10]{lucrare})\label{pan} 
We have
$$ \paa(n) = \frac{1}{(r-1)!} \sum_{\substack{0\leq j_1\leq \frac{D}{a_1}-1,\ldots, 0\leq j_r\leq \frac{D}{a_r}-1 \\ 
a_1j_1+\cdots+a_rj_r \equiv n (\bmod D)}} \prod_{\ell=1}^{r-1} \left(\frac{n-a_1j_{1}- \cdots -a_rj_r}{D}+\ell \right).$$
\end{teor}

The \emph{unsigned Stirling numbers} are defined by
\begin{equation}
 \binom{n+r-1}{r-1}=\frac{1}{n(r-1)!}n^{(r)}=\frac{1}{(r-1)!}\left(\stir{r}{r}n^{r-1} + \cdots \stir{r}{2}n + \stir{r}{1}\right).
\end{equation}

\begin{teor}(\cite[Proposition 4.2]{remarks})\label{unde}
For any positive integer $j$ with $j|a_i$ for some $1\leq i\leq r$, we have that
\begin{align*}
 W_{j}(n,\mathbf a) = \frac{1}{D(r-1)!} \sum_{m=1}^r \sum_{\ell=1}^{j} \rho_j^{\ell} \sum_{k=m-1}^{r-1} 
\stir{r}{k+1} (-1)^{k-m+1} \binom{k}{m-1} \times \\
\times \sum_{\substack{0\leq j_1\leq \frac{D}{a_1}-1,\ldots, 0\leq j_r\leq \frac{D}{a_r}-1 \\ a_1j_1+\cdots+a_rj_r \equiv \ell (\bmod j)}} D^{-k} (a_1j_1+\cdots+a_rj_r)^{k-m+1} n^{m-1}.
\end{align*}
\end{teor}

\begin{teor}(\cite[Corollary 3.6]{lucrare})\label{Pan}
For the polynomial part $\Pa(n)$ of the quasi-polynomial $\paa(n)$ we have
$$ \Pa(n) = \frac{1}{D(r-1)!} \sum_{0\leq j_1\leq \frac{D}{a_1}-1,\ldots, 0\leq j_r\leq \frac{D}{a_r}-1} 
\prod_{\ell=1}^{r-1} \left(\frac{n-a_1j_{1}- \cdots -a_rj_r}{D}+\ell \right).$$
\end{teor}

The \emph{Bernoulli numbers} $B_{\ell}$'s are defined by 
$\frac{t}{e^t-1}=\sum_{\ell=0}^{\infty}\frac{t^{\ell}}{\ell !}B_{\ell}$. We have
$B_0=1$, $B_1 = -\frac{1}{2}$, $B_2=\frac{1}{6}$, $B_4=-\frac{1}{30}$ and $B_n=0$ is $n$ is odd and $n\geq 1$.

\begin{teor}(\cite[Corollary 3.11]{lucrare} or \cite[page 2]{beck})\label{Pan2}
The polynomial part of $\paa(n)$ is
$$P_{\mathbf a}(n) := \frac{1}{a_1\cdots a_r}\sum_{u=0}^{r-1}\frac{(-1)^u}{(r-1-u)!}\sum_{i_1+\cdots+i_r=u} 
\frac{B_{i_1}\cdots B_{i_r}}{i_1!\cdots i_r!}a_1^{i_1}\cdots a_r^{i_r} n^{r-1-u}.$$
\end{teor}

\section{New formulas for the number of d-ary partitions}

We fix $d\geq 2$ an integer. We denote $\mathcal P$, the set of integer partitions, and $\mathcal P_d$, the set of $d$-ary partitions.
Given a positive integer $n$, we denote $p_d(n)$, the number of $d$-ary partitions of $n$.

\begin{dfn}
Let $\lambda=(\lambda_1,\ldots,\lambda_{\ell})\in\mathcal P$ be a partition. The $d$-exponential of $\lambda$ is the
$d$-ary partition: 
$$\Exp_d(\lambda):=(d^{\lambda_1-1},\ldots,d^{\lambda_{\ell}-1}).$$
\end{dfn}

\begin{dfn}
Let $\lambda=(\lambda_1,\ldots,\lambda_{\ell})\in\mathcal P_d$ be a $d$-ary partition. The $d$-logarithm of $\lambda$ is
the partition: 
$$\Log_d(\lambda):=(\log_d(\lambda_1)+1,\ldots,\log_d(\lambda_{\ell})+1).$$
\end{dfn}

\begin{prop}\label{p1}
The maps $\Exp_d:\mathcal P \to \mathcal P_d$ and $\Log_d:\mathcal P_d \to \mathcal P$ are bijective and inverse of each other.
\end{prop}

\begin{proof}
Let $\lambda=(\lambda_1,\ldots,\lambda_{\ell})\in\mathcal P$. We have $\Exp_d(\lambda)=(d^{\lambda_1-1},\ldots,d^{\lambda_{\ell}-1})$.
Since $$\log_d(d^{\lambda_i-1})+1=\lambda_i-1+1=\lambda_i\text{ for all }1\leq i\leq \ell,$$ it follows that 
$\Log_d(\Exp_d(\lambda))=\lambda$. 
Similary, if $\mu\in \mathcal P_d$ is a $d$-ary partition, then it is easy to see that $\Exp_d(\Log_d(\mu))=\mu$. Hence, the proof is complete.
\end{proof}

\begin{lema}\label{lem31}
Let $n$ and $k$ be two positive integers such that $n<d^{k+1}$.
The number of $d$-ary partitions of $n$
is $$p_d(n)=p_{(1,d,\ldots,d^{k})}(n).$$
In particular, the polynomial part of $p_d(n)$ is $P_d(n)=P_{(1,d,\ldots,d^{k})}(n)$.
\end{lema}

\begin{proof}
Let $\lambda=(\lambda_1,\ldots,\lambda_{\ell})$ be a $d$-ary partition of $n$, that is $n=|\lambda|$.
It follows that $\lambda_i=d^{c_i}$ with $0\leq c_i$ and $d^{c_i}\leq n$ for all $1\leq i\leq \ell$.
Since $\lambda_1 = d^{c_1} \leq |\lambda| < d^{k+1}$ and $\lambda_1\geq \lambda_2\geq \cdots \geq \lambda_{\ell}$,
it follows that $$k\geq c_1\geq c_2\geq \cdots \geq c_{\ell}\geq 0,$$ 
and, therefore, $\lambda$ is a partition with parts in $(1,d,\ldots,d^k)$. 
On the other hand, any partition with parts in $(1,d,\ldots,d^k)$ is a $d$-ary partition.
Hence, the proof is complete.
\end{proof}

\begin{teor}\label{t1}
Let $n$ and $k$ be two positive integers such that $n<d^{k+1}$.
The number of $d$-ary partitions of $n$ is \small
$$p_d(n)=\frac{1}{k!} \sum_{\substack{0\leq j_1\leq d^{k}-1,\; 0\leq j_2\leq d^{k-1}-1,\; \ldots, 0\leq j_{k} \leq d-1 \\ 
j_1+j_2d+\cdots+j_kd^{k-1} \equiv n (\bmod d^k)}} \prod_{\ell=1}^{k} \left(\frac{n-j_{1}- j_2 d- \cdots -j_k d^{k-1}}{d^k}+\ell \right).$$
\normalsize
\end{teor}

\begin{proof}
According to Lemma \ref{lem31}, we have $p_d(n)=p_{(1,d,\ldots,d^k)}(n)$, where $k=\lfloor \log_d(n) \rfloor$. 
Hence, the conclusion follows from Theorem \ref{pan}, taking $r=k+1$ and $D=\lcm(1,d,\ldots,d^k)=d^k$.
\end{proof}

From Lemma \ref{lem31} and \eqref{wave} we can write
$$p_d(n)=\sum_{j\geq 1}W_j(d,n),\text{ where }W_j(d,n)=W_j(n,(1,d,\ldots,d^k)),$$
and $k=\lfloor \log_d(n) \rfloor$. In particular, the polynomial part of $p_d(n)$ is
$$P_d(n)=W_1(d,n).$$

\begin{teor}\label{t2}
Let $n$ and $k$ be two positive integers such that $n<d^{k+1}$. We have that
\begin{align*}
W_j(d,n)=\frac{1}{k!d^k} \sum_{m=1}^{k+1} \sum_{\ell=1}^{j} \rho_j^{\ell} \sum_{s=m-1}^{k} 
\stir{k+1}{s+1} (-1)^{s-m+1} \binom{s}{m-1} \times \\
\times \sum_{\substack{0\leq j_1\leq d^k-1,\ldots, 0\leq j_k\leq d-1 \\ j_1+dj_2+\cdots+d^{k-1}j_{k-1} \equiv \ell (\bmod j)}} d^{-ks} 
(j_1+dj_2+\cdots+d^{k-1}j_{k-1})^{s-m+1} n^{m-1}.
\end{align*}
\end{teor}

\begin{proof}
The conclusion follows from Lemma \ref{lem31} and Theorem \ref{unde}.
\end{proof}

\begin{teor}\label{t3}
Let $n$ and $k$ be two positive integers such that $n<d^{k+1}$. 
The polynomial part of $p_d(n)$ is
\small
$$P_d(n)=\frac{1}{k!d^k} \sum_{0\leq j_1\leq d^{k}-1,\; 0\leq j_2\leq d^{k-1}-1,\; \ldots, 0\leq j_{k} \leq d-1} 
\prod_{\ell=1}^{k} \left(\frac{n-j_{1}- j_2 d- \cdots -j_k d^{k-1}}{d^k}+\ell \right).$$
\normalsize
\end{teor}

\begin{proof}
The conclusion follows from Lemma \ref{lem31} and Theorem \ref{Pan}.
\end{proof}

\begin{teor}\label{t4}
Let $n$ and $k$ be two positive integers such that $n<d^{k+1}$. 
The polynomial part of $p_d(n)$ is
\small
$$P_d(n)= \frac{1}{d^{\frac{k(k+1)}{2}}} \sum_{u=0}^{k} \frac{(-1)^u}{(k-u)!}\sum_{i_1+\cdots+i_{k+1}=u} 
\frac{B_{i_1}\cdots B_{i_{k+1}}}{i_1!\cdots i_{k+1}!}d^{i_2+2i_3+\cdots+ki_{k+1}} n^{k-u}.$$
\normalsize
\end{teor}

\begin{proof}
The conclusion follows from Lemma \ref{lem31} and Theorem \ref{Pan2}.
\end{proof}

\begin{exm}\label{exe}\rm
(1) Let $n=8$ and $d=3$. Since $n<d^{1+1}$, Theorem \ref{t1} implies \small
$$p_3(8)=\frac{1}{1!}\sum_{0\leq j_1\leq 2,\;j_1\equiv 8(\bmod 3)} \left(\frac{8-j_1}{3}+1 \right) = \frac{8-2}{3}+1 = 3.$$
\normalsize
Also, from Theorem \ref{t3} it follows that the polynomial part of $p_3(8)$ is
\small
$$P_3(8)=\frac{1}{1!\cdot 3^1}\sum_{j_1=0}^2 \left(\frac{8-j_1}{3}+1 \right) = \frac{1}{9}\sum_{j_1=0}^2(11-j_1)=\frac{11+10+9}{9}=\frac{10}{3}.$$
\normalsize
(2) Let $n=20$ and $d=3$. Since $n<d^{2+1}$, Theorem \ref{t1} implies \small
$$p_3(20)=\frac{1}{162}\sum_{\substack{0\leq j_1\leq 8,\;0\leq j_2\leq 2\\ j_1+3j_2\equiv 20(\bmod 9)}} (29-j_1-3j_2)(38-j_1-3j_2).$$ \normalsize
Since the set of pairs $(j_1,j_2)$ which satisfy the above conditions is $\{(2,0),(8,1),(5,2)\}$, it follows that
$p_3(20)=\frac{1}{162}(28\cdot 36 + 18\cdot 27 + 18\cdot 27)=12$.
\end{exm}


\section{An application to elementary symmetric partitions}

Given $n\geq 2$ an integer, we denote by $\{e_1,\ldots,e_n\}$, the standard basis of the
vector space $\mathbb R^n$, i.e. $e_i$ is the vector with $1$ in the $i$-th position and zeros everywhere else.

Let $1\leq j\leq n-1$ be an integer. We consider the vectors:
$$c_i=\begin{cases} e_1+e_2+\cdots+e_j,& i=1 \\
      e_1+e_2+\cdots+e_{j+1}-e_{i-1},&2\leq i\leq j+1 \\ e_{i-j+1}+e_{i-j+2}+\cdots+e_{i},& j+2\leq i\leq n\end{cases}$$
Let $C$ be the $n\times n$ matrix whose columns are $c_1,c_2,\ldots,c_n$.

To better illustrate the structure of the matrix $C$, we present the case $n=6$ and $j=3$:
$$C = \begin{pmatrix} 1 & 0 & 1 & 1 & 0 & 0 \\ 
                      1 & 1 & 0 & 1 & 0 & 0 \\
											1 & 1 & 1 & 0 & 1 & 0 \\
											0 & 1 & 1 & 1 & 1 & 1 \\
											0 & 0 & 0 & 0 & 1 & 1 \\
											0 & 0 & 0 & 0 & 0 & 1 \end{pmatrix}.$$

\begin{lema}\label{lem}
With the above notations, we have that $\det(C)=j$.
\end{lema}

\begin{proof}
From the definition of $C$, we easily note that $\det(C)=\det(A)$, where
$$A=\begin{pmatrix} 1 & 0 & 1 & \cdots & 1 \\
                    1 & 1 & 0 & \cdots & 1 \\
						\vdots & \vdots & \vdots & \ddots & \vdots \\
						        1 & 1 & 1 & \cdots & 0 \\
										0 & 1 & 1 & \cdots & 1
													\end{pmatrix}$$
is a $(j+1)\times(j+1)$ circulant matrix with the associated polynomial 
$$f(x)=1+x+x^2+\cdots+x^{j-1}.$$
For more details on circulant matrices, we refer the reader to \cite{ingleton}.

Let $\omega=e^{\frac{2\pi i}{j+1}}$ be a primitive $(j+1)$-th root of unity. Using a basic result on
circulant matrices, we have that
$$\det(A)=\prod_{k=0}^{j} f(\omega^k).$$
It is clear that $f(\omega^0)=f(1)=j$. On the other hand, for $1\leq k\leq j$, we have that
$$f(\omega^k)=1+\omega^k+\cdots+\omega^{k(j-1)}=-\omega^{kj}.$$
Therefore, it follows that
$$\det(A)=(-1)^{j}j\omega^{\frac{j^2(j+1)}{2}}.$$
If $j$ is even, then $$\omega^{\frac{j^2(j+1)}{2}}=(\omega^{j+1})^{\frac{j^2}{2}}=1^{\frac{j^2}{2}} =1.$$
On the other hand, if $j$ is odd, then
$$\omega^{\frac{j^2(j+1)}{2}}=(\omega^{\frac{(j+1)}{2}})^{j^2}=(-1)^{j^2}=-1.$$
Hence, in both cases, we have that $\det(A)=j$. Thus, the proof is complete.
\end{proof}

\begin{teor}\label{tp}
Let $\lambda$ and $\mu$ be two $d$-ary partitions with $\ell$ parts and let $1\leq j\leq \ell-1$ be an integer.
If $\pre_j(\lambda)=\pre_j(\mu)$ and and $\lambda_{i_1}\cdots \lambda_{i_j} = 
\mu_{i_1}\cdots \mu_{i_j}$, for all $1\leq i_1 < \cdots < i_j\leq \ell$, then $\lambda=\mu$.
\end{teor}

\begin{proof}
Since $\lambda$ is a $d$-ary partition, it follows that 
$\lambda=(\lambda_1,\lambda_2,\ldots,\lambda_{\ell})$ such that $\lambda_i=d^{c_i}$, for all $1\leq i\leq \ell$, 
and $c_1\geq c_2\geq \cdots \geq c_{\ell}$. Similarly, $\mu=(\mu_1,\ldots,\mu_{\ell})$ with $\mu_i=d^{c'_i}$, 
for all $1\leq i\leq \ell$, and $c_1\geq c_2\geq \cdots \geq c_{\ell}$.

From the definition, $\pre_j(\lambda)$ is the partition whose parts are:
$$ \{ d^{c_{i_1}+c_{i_2}+\cdots+c_{i_j}}\;:\;1\leq i_1 < i_2 < \cdots < i_j\leq \ell \}.$$
Similarly, $\pre_j(\mu)$ is the partition whose parts are:
$$ \{ d^{c'_{i_1}+c'_{i_2}+\cdots+c'_{i_j}}\;:\;1\leq i_1 < i_2 < \cdots < i_j\leq \ell \}.$$
Since $\pre_j(\lambda)=\pre_j(\mu)$ and $\lambda_{i_1}\cdots \lambda_{i_j} = 
\mu_{i_1}\cdots \mu_{i_j}$, for all $1\leq i_1 < \cdots < i_j\leq \ell$, it follows that
$$ c'_{i_1}+c'_{i_2}+\cdots+c'_{i_j} = c_{i_1}+c_{i_2}+\cdots+c_{i_j},\text{for all }1\leq i_1 < i_2 < \cdots < i_j\leq \ell. $$
For convenience, we denote
$$c_{i_1,\ldots,i_j}:=c_{i_1}+c_{i_2}+\cdots+c_{i_j},\text{for all }1\leq i_1 < i_2 < \cdots < i_j\leq \ell.$$
From Proposition \ref{p1}, in order to prove that $\lambda=\mu$, it suffices to show that 
$(c_1,\ldots,c_{\ell})=(c'_1,\ldots,c'_{\ell})$. In order to do that, it is enough to prove that the linear system
\begin{equation}\label{sistem}
\begin{cases} x_{i_1}+x_{i_2}+\cdots+x_{i_j} = c_{i_1,\ldots,i_j} \end{cases},\text{where }1\leq i_1 < i_2 < \cdots < i_j\leq \ell,
\end{equation}
has a unique solution. Since $(c_1,\ldots,c_{\ell})$ is already a solution of \eqref{sistem}, it is enough to prove that the matrix associated to \eqref{sistem}
has the rank $n$. We consider the following subsystem of \eqref{sistem}:
\begin{equation}\label{ssistem}
\begin{cases} x_1+x_2+\cdots+x_j = c_{1,2,\ldots,j} \\
              x_2+x_3+\cdots+x_{j+1}=c_{2,\ldots,j+1} \\
							x_1+x_3+\cdots+x_{j+1}=c_{1,3,\ldots,j+1} \\
							\vdots \\
							x_1+\cdots+x_{j-1}+x_{j+1}=c_{1,\ldots,j-1,j+1}\\
							x_3+x_4+\cdots+x_{j+2} = c_{3,\ldots,j+2}\\
							x_4+x_5+\cdots+x_{j+3} = c_{4,\ldots,j+3}\\
							\vdots \\
							x_{\ell-j+1}+\cdots+x_{\ell}=c_{\ell-j+1,\ldots,\ell}
\end{cases}.
\end{equation}
Note that the matrix associated to \eqref{ssistem} is $C^T$, where $C$ was defined at the beginning of this section.

According to Lemma \ref{lem} we have $\det(C^T)=\det(C)=j\neq 0$. Hence, \eqref{ssistem} has a unique solution.
Thus \eqref{sistem} has also a unique solution, as required.
\end{proof}

\section{Conclusions}

Let $n\geq 1$ and $d\geq 2$ be two integers. We proved new formulas for $\pp_d(n)$, the number of $d$-ary partitions of $n$, and, also,
for $P_d(n)$, its polynomial part.

Given $\lambda$ a partition of length $\ell$ and $1\leq j\leq \ell-1$, we denote $\pre_j(\lambda)$, its associated $j$-th 
elementary symmetric partition; see \cite{merca,merca2}. Given $\lambda$ and $\mu$ two $d$-ary partitions of length $\ell$ and 
$1\leq j\leq \ell-1$, we proved that if $\pre_j(\lambda)=\pre_j(\mu)$ and $\lambda_{i_1}\cdots \lambda_{i_j} = 
\mu_{i_1}\cdots \mu_{i_j}$, for all $1\leq i_1 < \cdots < i_j\leq \ell$, then $\lambda=\mu$, thus giving a partial positive answer 
to a problem raised in \cite{merca}.

\end{document}